\documentclass{amsart}
\usepackage{amsmath,amsfonts,amssymb}

\usepackage{mathrsfs}

\newtheorem{thm}{Theorem}
\newtheorem{lem}[thm]{Lemma}

\begin{document}

\title[The character table of a split extension \ldots]{The character table of a split extension of the Heisenberg group $H_1(q)$ by $Sp(2,q)$, $q$ odd}
\author{Marco Antonio Pellegrini}
\address{Dipartimento di Matematica e Applicazioni,\\ 
Universit\`a degli Studi di Milano-Bicocca, \\
Via R. Cozzi, 53, \\
20125 Milano (Italy)}
\email{marco.pellegrini@unimib.it}



\begin{abstract}
In this paper we determine the full character table of a certain split extension $H_1(q)\rtimes Sp(2,q)$ of the Heisenberg group $H_1$  by the odd-characteristic symplectic group $Sp(2,q)$.
\end{abstract}

\maketitle

\section{Introduction}

In his paper (\cite{Ger}) P. G\'erardin constructed the Weil representations of the odd-characteristic symplectic groups using the properties of a certain split extension $H_t(q)\rtimes Sp(2t,q)$ of the Heisenberg group $H_t(q)$ of order $q^{2t+1}$ by the symplectic group $Sp(2t,q)$. In this paper we explicitly determine the character table of this extension, in the case where $t=1$. A motivation lies in the fact that knowledge of this character table seems to be useful in the study of the restrictions to parabolic subgroups of certain unipotent characters of odd-dimensional orthogonal groups (see \cite{DPW}).

Let $V$ be the column vector space of dimension $2t$ over a finite field $F$ of order $q$, where $q$ is odd, and $V$ is provided with a non-degenerate symplectic form $j$. Given $w\in V$, we denote by $w^*$ the element of the dual space (we think at $w^*$ as a row) such that $w^* w_1= j(w,w_1)/2$. Let $H_t(q)$ be the group consisting of the matrices
$$h=h_{(w,z)}=\left(\begin{array}{c|c|c}
1 & w^* & z \\ \hline
{} & 1 & w \\ \hline
{} & {} & 1 
\end{array}
 \right) \in Mat(2t+2,F),$$
where $w\in V$ and $z\in F$. We call this group the Heisenberg group of $V$. $H_t(q)$ is obviously a central extension of $(V,+)$ by $(F,+)$. Furthermore, $H_t(q)$ is a two-step nilpotent group of order $q^{2t+1}$ whose center is isomorphic to $F$ (cf. \cite[Lemma 2.1]{Ger}).

Let $S$ be the symplectic group associated to the form $j$ and, for each $s\in S$, denote by $sw$ the image of $w$ under the natural action of $S$ on $V$. Then, the map $h_{(w,z)}\mapsto h_{(sw,z)}$  defines an automorphism of $H_t(q)$ fixing pointwise $\mathbf{Z}(H_t(q))$. Viewed as acting on matrices, this map is the conjugation by the element $\mathbf{s}=diag(1,s,1)$

Let us denote by $G$ the semidirect product $H_t(q)\rtimes Sp(2t,q)$ defined by the above action of $S$. We want to construct the character table of $G$ in the case where $t=1$. So, $G=H_1(q)\rtimes Sp(2,q)$. In this case, we can write in a unique way a generic element $g$ of $G$ as 
$$g=g_{(s,w,z)}=sh_{(w,z)}=\left(\begin{array}{c|c|c}
1 & w^* & z \\ \hline
{} & s & sw \\ \hline
{} & {} & 1 
\end{array}
 \right) ,$$
\noindent where $s\in S=Sp(2,q)$ (here we identify $s\in S$ with $\mathbf{s}\in G$ ), $w\in V$ and $z\in F$. If $w=\left(\begin{array}{c}
x\\ 
y
\end{array} \right) \in V$, then we can take as $w^*$ the row  $\frac{1}{2}(-y,x)$. Note that $|G|=q^4(q^2-1)$.

\section{The conjugacy classes}

In the sequel, we denote by $(g)$ the conjugacy class of $G$ containing the element $g$, and by $|(g)|$ the size of the conjugacy class $(g)$. The following lemma lists the conjugacy classes of $G$.

\begin{lem}
Let $F=GF(q)$, $q$ odd, and let $F^\times=\langle\nu\rangle$ be the multiplicative group of $F$. Set
$$
\mathscr{A}(z)=\left(
\begin{array}{c|cc|c}
1 & {} & {} & z\\\hline
{} & 1 & {} & {} \\
{} & {} & 1 & {} \\\hline
{} & {} & {} & 1
\end{array}
\right), \quad\quad 
\mathscr{B}=\left(
\begin{array}{c|cc|c}
1 & {} & \frac{1}{2} & {}\\\hline
{} & 1 & {} & 1 \\
{} & {} & 1 & {} \\ \hline
{} & {} & {} & 1
\end{array}
\right) ,\\
$$
$$
\mathscr{C}(z)=\left(
\begin{array}{c|cc|c}
1 & {} & {} & z\\\hline
{} & -1 & {} & {} \\
{} & {} & -1 & {} \\\hline
{} & {} & {} & 1
\end{array}
\right), \quad\quad 
\mathscr{D}_k(z)=\left(
\begin{array}{c|cc|c}
1 & {} & {} & z\\\hline
{} & \nu^k & {} & {} \\
{} & {} & \nu^{-k} & {} \\\hline
{} & {} & {} & 1
\end{array}
\right) ,\\
$$
$$
\mathscr{E}(z)=\left(
\begin{array}{c|cc|c}
1 & {} & {} & z\\\hline
{} & -1 & {} & {} \\
{} & -1 & -1 & {} \\\hline
{} & {} & {} & 1
\end{array}
\right), \quad\quad 
\mathscr{F}(z)=\left(
\begin{array}{c|cc|c}
1 & {} & {} & z\\\hline
{} & -1 & {} & {} \\
{} & -\nu & -1 & {} \\\hline
{} & {} & {} & 1
\end{array}
\right) ,
$$
$$
\mathscr{G}_m(z)=\left(
\begin{array}{c|c|c}
1 & {} & z\\\hline
{} & \mathbf{b}^m & {} \\\hline
{} & {} & 1
\end{array}
\right), \quad\quad 
\mathscr{H}(z)=\left(
\begin{array}{c|cc|c}
1 & {} & {} & z\\\hline
{} & 1 & {} & {} \\
{} & 1 & 1 & {} \\\hline
{} & {} & {} & 1
\end{array}
\right) ,
$$
$$
\mathscr{I}(z)=
\left(
\begin{array}{c|cc|c}
1 & {} & {} & z\\\hline
{} & 1 & {} & {} \\
{} & \nu & 1 & {} \\\hline
{} & {} & {} & 1
\end{array}
\right), \quad\quad
\mathscr{L}_m=\left(
\begin{array}{c|cc|c}
1 & 0 & \frac{1}{2}\nu^m & {}\\\hline
{} & 1 & {} & \nu^m \\
{} & 1 & 1 & \nu^m \\\hline
{} & {} & {} & 1
\end{array}
\right) ,$$
$$ 
\mathscr{M}_m=\left(
\begin{array}{c|cc|c}
1 & 0 & \frac{1}{2}\nu^m & {}\\\hline
{} & 1 & {} & \nu^m \\
{} & \nu & 1 & \nu^{m+1} \\\hline
{} & {} & {} & 1
\end{array}
\right) ,
$$

\noindent where $z\in F$, $1\leq k\leq \frac{q-3}{2}$, $1\leq m\leq \frac{q-1}{2}$ and $\mathbf{b}$ is an element of order $q+1$ (a `Singer cycle') of $Sp(2,q)$. These are elements of $G$, and $G$ admits exactly $q^2+5q$ conjugacy classes $(g)$ with representative $g$, as listed in the Table below.

\begin{center}
\begin{tabular}{|c|c|c|}\hline
$g$ & $|(g)|$ & Parameters\\ \hline\hline
$\mathscr{A}(z)$ & 1 & $z\in F$ \\
$\mathscr{B}$ & $q(q^2-1)$ & {}\\
$\mathscr{C}(z)$ & $q^2$ & $z\in F$ \\
$\mathscr{D}_k(z)$ & $q^3(q+1)$ & $z\in F$, $1\leq k\leq\frac{q-3}{2}$ \\
$\mathscr{E}(z)$ & $\frac{1}{2}q^2(q^2-1)$ & $z\in F$ \\
$\mathscr{F}(z)$ & $\frac{1}{2}q^2(q^2-1)$ & $z \in F$ \\
$\mathscr{G}_m(z)$ & $q^3(q-1)$ & $z\in F$, $1\leq m\leq \frac{q-1}{2}$ \\
$\mathscr{H}(z)$ & $\frac{1}{2}q(q^2-1)$ & $z \in F$ \\
$\mathscr{I}(z)$ & $\frac{1}{2}q(q^2-1)$ & $z\in F$ \\
$\mathscr{L}_m$ & $q^2(q^2-1)$ & $1\leq m\leq \frac{q-1}{2}$ \\
$\mathscr{M}_m$ & $q^2(q^2-1)$ & $1\leq m\leq \frac{q-1}{2}$ \\ \hline
\end{tabular}
\end{center}
\end{lem}

\begin{proof}
Let $g_1=g_{(s_1,w_1,z_1)}$ and $g_2=g_{(s_2,w_2,z_2)}$ be two generic elements of $G$. Then
$g_1 g_2 g_1^{-1} =g_{(s_1 s_2 s_1^{-1}, s_1(w_2-w_1+s_2^{-1} w_1),z_2-(w_2+s_2^{-1}w_1+s_2 w_2)^*w_1)}$. It easily follows  that if $g_1$ is conjugate to $g_2$ in $G$, then $s_1$ is conjugate to $s_2$ in $S$. Moreover,  if $z_1\neq z_2$, then the elements $g_{(s_1,0,z_1)}$ and $g_{(s_2,0,z_2)}$ cannot be conjugate in $G$.
Observe that $g_1 \in \mathbf{C}_G(g_2)$ if and only if 
$$\left\{
\begin{array}{l}
s_1\in \mathbf{C}_S(s_2) \\
w_2+s_2^{-1}w_1 = w_1+s_1^{-1}w_2 \\
w_1^*(s_2 w_2) = w_2^*(s_1w_1) \\
\end{array}\right. .$$\\

Let us consider the elements $\mathscr{A}(z)=g_{(1,0,z)}$, $z\in F$. It is straightforward to see that $\mathbf{Z}(G)=\mathbf{Z}(H_1(q))= \{\mathscr{A}(z): z \in F \}\cong (F,+)$. Therefore, each of these $q$ elements of $G$ forms a central class of order 1. In particular, $\mathscr{A}(0)$ is the identity of $G$. \\

Now, let us consider the element $g_{(1,w,0)}=\mathscr{B} \in H_1(q)\setminus \mathbf{Z}(H_1(q))$. Then, $|\mathbf{C}_{G}(\mathscr{B})|=q^3$, i.e. $|(\mathscr{B})|=q(q^2-1)$. Since
$$g_{(s_1,w_1,z_1)} \mathscr{B} g_{(s_1,w_1,z_1)}^{-1}=g_{(1,s_1w,-2w^*w_1)},$$
it turns out that the elements of $H_1(q)\setminus \mathbf{Z}(H_1(q))$ form a single conjugacy class $(\mathscr{B})$ of $G$.\\

Set $g=g_{(s,0,z)}\in \{\mathscr{C}(z), \mathscr{D}_k(z), \mathscr{E}(z), \mathscr{F}(z), \mathscr{G}_m(z)\}$. Recall (e.g., see \cite[\S 38]{Dor}) that $S$ admits elements $\mathbf{b}$ of order $q+1$, the so-called `Singer cycles'. As observed before, for different values of $z$ and $s$ the elements $g_{(s,0,z)}$ belong to $q^2+q$ distinct conjugacy classes of $G$. Now, an element $g_{(s_1,w_1,z_1)}$ belongs to $\mathbf{C}_G(g)$ if and only if 
\begin{equation}\label{eq 1}
\left\{\begin{array}{l}
s_1 \in \mathbf{C}_S(s)\\
s w_1=w_1
\end{array}
\right. .
\end{equation}

Since $s$ does not have eigenvalue $1$, the condition $s w_1=w_1$ implies $w_1=0$. It follows that $|\mathbf{C}_G(g)|= q|\mathbf{C}_S(s)|$, and using the information about the centralizers of elements of $S$ contained in \cite[\S 38]{Dor}, we obtain the results listed in the statement of the lemma.\\

Next, let us consider elements $g=g_{(s,0,z)}\in \{\mathscr{H}(z), \mathscr{I}(z)\}$. We argue as above, but note that this time $s$ does admit the eigenvalue $1$. This implies that in (\ref{eq 1}) $w_1=\left(
\begin{array}{c}
0 \\
y \\
\end{array} \right),$ where $y \in F$. So $|\mathbf{C}_G(g)|=q^2|\mathbf{C}_S(s)|=2q^3$, i.e. $|(g)|=\frac{q(q^2-1)}{2}$.\\

Finally, let us consider elements $g=g_{(s,w,0)}\in \{\mathscr{L}_m, \mathscr{M}_m\}$, where 
$$s=\left(
\begin{array}{cc}
1 & 0 \\
\epsilon & 1\\
\end{array} \right), \quad
w=\left(
\begin{array}{cc}
\nu^m\\
0 \\
\end{array} \right) ,$$ 
and $\epsilon \in \{1, \nu\}$. Easy calculations show that if $1\leq m\leq \frac{q-1}{2}$ the elements $g$ belong to distinct conjugacy classes of $G$. An element $g_{(s_1,w_1,z_1)}$ belongs to $\mathbf{C}_G(g)$ if and only if 
$$\left\{\begin{array}{l}
s_1 \in \mathbf{C}_S(s)= \Big\{\left( \begin{array}{cc} a & 0\\ c & a \end{array}\right) : a=\pm 1,  c\in F \Big\} \\
w+s^{-1}w_1=w_1+s_1^{-1}w\\
w_1^*(sw)=w^*(s_1w_1)
\end{array}
\right. .$$
Since the condition $w+s^{-1}w_1= w_1+s_1^{-1}w$ implies $a=1$, it follows that $g_{(s_1,w_1,z_1)}$ can be chosen in $q^2$ different ways. Thus, $|\mathbf{C}_G(g)|=q^2$, i.e. $|(\mathscr{L}_m)|=|(\mathscr{M}_m)| =q^2(q^2-1)$.\\

So far, we have found $q^2+5q$ distinct conjugacy classes, adding up to $|G|$ elements. Thus, we are done. 
\end{proof}

\section{The character table}

First of all, we observe that the character table of $SL(2,q)\cong Sp(2,q)\cong G/H_1(q)$ is well-known, e.g., see \cite[\S 38]{Dor}, to which we refer for notation and all the information needed in the sequel.

Next, note that, as $\mathbf{Z}(G)=\{\mathscr{A}(z): z \in F\}$, for any irreducible character $\chi$ of $G$
$$\chi(\mathscr{C}(z))=\frac{\chi(\mathscr{A}(z))}{\chi(1)}\chi(\mathscr{C}(0))$$
for all $z \in F$. The same holds for the classes $(\mathscr{D}_k(z))$, $(\mathscr{E}(z))$, $(\mathscr{F}(z))$, $(\mathscr{G}_m(z))$, $(\mathscr{H}(z))$ and $(\mathscr{I}(z))$. So, in the character table we only report the values of a character on $\mathscr{C}(0)$, $\mathscr{D}_k(0)$ and so on.\\

Since $G/H_1(q)\cong SL(2,q)$, knowledge of the character table of $SL(2,q)$ gives us by inflation $q+4$ characters: namely $1_G$, $\eta_1$, $\eta_2$, $\xi_1$, $\xi_2$, $\theta_j$ ($1\leq j \leq \frac{q-1}{2}$), $\psi$ and $\chi_i$ ($1\leq i\leq \frac{q-3}{2}$).\\

Next, we construct $q-1$ distinct irreducible characters of $G$ having degree $q$. Denote by $\lambda$ a fixed non-trivial character of $\mathbf{Z}(G)\cong (F,+)$. Clearly, each of the $q$ linear characters of $\mathbf{Z}(G)$ can be parametrised as $\lambda_u$ ($u\in F$), where $\lambda_u(z)=\lambda(uz)$ for all $z \in F$. In particular, $\lambda_0=1_{\mathbf{Z}(G)}$. We know by \cite[Lemma 1.2]{Ger} that $H_1(q)$ has exactly $q-1$ non-linear irreducible characters $\tilde\lambda_u$, defined as
$$\tilde\lambda_u(h)=\left\{ \begin{array}{cl}
q\lambda_u(h) & \textrm{ if } h \in \mathbf{Z}(H_1(q))  \\
0 & \textrm{ if } h \not\in \mathbf{Z}(H_1(q))
\end{array}\right. \quad (u\in F^\times).$$

Furthermore, by \cite[Theorem 2.4]{Ger} the characters $\tilde\lambda_u$ can be extended to $G$. We denote such extensions by $\omega_u$ ($u \in F^\times$). The values taken by the characters $\omega_u$ on the elements of $S$ can be found in \cite[Proposition 2]{Sz}. Namely:

\begin{small}
\begin{center}
\begin{tabular}{|c||c|c|c|c|c|c|c|c|c|c|} \hline
$g$ & $1$ & $\mathscr{A}(z)$ & $\mathscr{B}$ & $\mathscr{C}(0)$ & $\mathscr{D}_k(0)$ & $\mathscr{E}(0)$ & $\mathscr{F}(0)$ & $\mathscr{G}_m(0)$ & $\mathscr{H}(0)$ & $\mathscr{I}(0)$ \\ \hline 
$\omega_u(g)$ & $q$ & $q\lambda_u(z)$ & $0$ & $\delta$ & $(-1)^k$ & $\delta$ & $\delta$ & $(-1)^{m+1}$ & $Q(\lambda_u)$  & $-Q(\lambda_u)$\\ \hline
\end{tabular}
\end{center}
\end{small}

\noindent where

$$Q(\lambda)=\sum_{t\in F}\lambda(-t^2/2), \quad \quad Q(\lambda_u)=\sum_{t\in F}\lambda_u(-t^2/2)=\Big(\dfrac{u}{F}\Big)Q(\lambda)$$
and
$$\Big(\dfrac{u}{F}\Big)=\left\{\begin{array}{cl} 
+1 & \textrm{ if } u \textrm{ is a square in } F  \\
-1 &  \textrm{ if } u \textrm{ is not a square in } F
\end{array}\right. $$
(it turns out that $|Q(\lambda)|^2=q$).

We are left to compute the values of the $\omega_u$'s on the classes $(\mathscr{L}_m)$ and $(\mathscr{M}_m)$. To this purpose, we compute 
$$1=(\omega_u,\omega_u)_G=\frac{q^4(q^2-1)+q^2(q^2-1)\sum_{m=1}^{\frac{q-1}{2}} (|\omega_u(\mathscr{L}_m)|^2+|\omega_u(\mathscr{M}_m)|^2)}{q^4(q^2-1)}.$$
This implies that $\omega_u(\mathscr{L}_m)=\omega_u(\mathscr{M}_m)=0$, for all $1\leq m\leq \frac{q-1}{2}$.\\

It is easy to verify that the characters $\omega_u\eta_1$, $\omega_u\eta_2$, $\omega_u\xi_1$, $\omega_u\xi_2$, $\omega_u\theta_j$, $\omega_u\psi$ and $\omega_u\chi_i$ ($u \in F^\times$) are pairwise distinct irreducible characters of $G$.\\

At this stage, $q$ irreducible characters of $G$ are still missing. We construct them as follows.

Let us consider the Sylow $p$-subgroup $K$ of $G$ consisting of the matrices of shape

$$k_{(a,x,y,z)}=\left(
\begin{array}{c|cc|c}
1 & -y/2 & x/2 & z\\\hline
{} & 1 & a & x+ay \\
{} & {} & 1 & y \\\hline
{} & {} & {} & 1
\end{array}
\right) ,$$
where $a,x,y \in F$. Define the linear characters $\mu_{u_1,u_2}$ ($u_1,u_2\in F$) of $K$ setting
$\mu_{u_1,u_2}(k_{(a,x,y,z)})= \lambda_{u_1}(a)\lambda_{u_2}(y)=\lambda(u_1 a+u_2 y)$, where, as above, $\lambda_u$ denotes the non-trivial linear character of $\mathbf{Z}(G)$ associated to $u\in F^\times$ (in particular, $\mu_{0,0}=1_K$). 

We consider the induced characters $\mu_{u_1,u_2}^G$.

First of all, note that $(\mathscr{C}(z))\cap K=\emptyset$ and that the same holds also for $(\mathscr{D}_k(z))$, $(\mathscr{E}(z))$, $(\mathscr{F}(z))$ and $(\mathscr{G}_m(z))$. So, the value of $\mu_{u_1,u_2}^G$ on these classes is $0$, whereas the value on $\mathscr{A}(z)$ is
$$\mu_{u_1,u_2}^G(\mathscr{A}(z))= \frac{q^4(q^2-1)}{q^4}=q^2-1.$$

\noindent To compute $\mu_{u_1,u_2}^G(\mathscr{B})$, we observe that if $s=\left(\begin{array}{cc} s_{11} & s_{12} \\ s_{21} & s_{22} \end{array}\right)\in S$ and $g=g_{(s,w,z)}$, then $\mu_{u_1,u_2}(g\mathscr{B}g^{-1})=\lambda_{u_1}(0)\lambda_{u_2}(s_{21})=\lambda_{u_2}(s_{21})$. 
So, if $s_{21}=0$ the matrix $s$ can be chosen in $q(q-1)$ ways, whereas if we fix $s_{21}\neq 0$, $s$ can be chosen in $q^2$ different ways. For $u_2\neq 0$ we obtain 
\begin{eqnarray*}
\mu_{u_1,u_2}^G(\mathscr{B}) & = &\frac{q^3[q(q-1)+q^2\sum_{s_{21}\neq 0}\lambda_{u_2}(s_{21}) ]}{q^4} \\
{} & = &  \frac{q^3[q(q-1)-q^2]}{q^4}= -1\\
\end{eqnarray*}
Next, we look at the classes $(\mathscr{H}(z))$. The matrices $s$ such that $g\mathscr{H}(z)g^{-1} \in K$ are of shape $\left(\begin{array}{cc} s_{11} & s_{12} \\ -1/s_{12} & 0 \end{array}\right)$. It follows that
$$\mu_{u_1,u_2}(g\mathscr{H}(z)g^{-1})=\lambda_{u_1}(-s_{12}^2)\lambda_{u_2}(0)$$
and therefore
\begin{eqnarray*}
\mu_{u_1,u_2}^G(\mathscr{H}(z)) & = &\frac{q^3q\sum_{t\neq 0}\lambda_{u_1}(-t^2)}{q^4}\\
{} & = & -1+Q(\lambda_{2u_1})= -1+ \Big(\dfrac{2}{F}\Big) Q(\lambda_u)\\
\end{eqnarray*}
In a similar way, one also obtains that 
$$\mu_{u_1,u_2}^G(\mathscr{I}(z))=\sum_{t\neq 0}\lambda_{u_1}(-\nu t^2).$$

In particular, for $u_1=0$ we get $\mu_{0,u_2}^G(\mathscr{H}(z))=\mu_{0,u_2}^G(\mathscr{I}(z))=q-1$, whereas for $u_1\neq 0$, we get $\mu_{u_1,u_2}^G(\mathscr{I}(z))=-1-Q(\lambda_{2u_1})$.

The value of $\mu_{u_1,u_2}^G$ on $\mathscr{L}_m$ is obtained in the same way as above: $s$ has the same shape as in the case $\mathscr{H}(z)$, but $\mu_{u_1,u_2}(g\mathscr{L}_m g^{-1})= \lambda_{u_1}(-s_{12}^2)\lambda_{u_2}(\frac{-\nu^m}{s_{12}})$. Thus,

\begin{eqnarray*}
\mu_{u_1,u_2}^G(\mathscr{L}_m) & = &  \frac{q^3q\sum_{t\neq 0}\lambda_{u_1}(-t^2)\lambda_{u_2}(-\nu^m/t)}{q^4} \\
{} & = & \sum_{t\neq 0}\lambda\Big(-\frac{ u_1 t^3+u_2 \nu^m }{t}\Big) .
\end{eqnarray*}

\noindent Similarly, in the case of $\mathscr{M}_m$, we obtain
\begin{eqnarray*}
\mu_{u_1,u_2}^G(\mathscr{M}_m)  & = &  \frac{q^3q\sum_{t\neq 0}\lambda_{u_1}(-\nu t^2)\lambda_{u_2}(-\nu^m/t)}{q^4}\\
{} & = & \sum_{t\neq 0}\lambda\Big(-\frac{ u_1 \nu t^3+u_2 \nu^m }{t}\Big) .
\end{eqnarray*}
In particular, for $u_1=0$ we get $\mu_{0,u_2}^G(\mathscr{L}_m)=\mu_{0,u_2}^G(\mathscr{M}_m)=-1$.

Set $\kappa_0=\mu_{0,1}^G$. Computing $(\kappa_0,\kappa_0)_G$, one sees that $\kappa_0$ is irreducible. Furthermore, for all $u_1,u_2\in F^\times$, $\kappa_0$ is different from any of the $\mu_{u_1,u_2}^G$'s  because 
$$\kappa_0(\mathscr{H}(0))+\kappa_0(\mathscr{I}(0))=2q-2\neq \mu_{u_1,u_2}^G(\mathscr{H}(0))+\mu_{u_1,u_2}^G(\mathscr{I}(0))=-2.$$

Next, we show that we can always pick $q-1$ pairwise distinct irreducible characters among the $\mu_{u_1,u_2}^G$'s. For instance, we can take as $(u_1,u_2)$ the pairs $(1,\nu^n)$ and $(\nu,\nu^n)$, where $1\leq n\leq \frac{q-1}{2}$. Set $\kappa_{1,n}=\mu_{1,\nu^n}^G$, $\kappa_{\nu,n}=\mu_{\nu,\nu^n}^G$. We start showing that these characters are irreducible.

Use of Mackey's formula implies that
$$(\kappa_{1,n},\kappa_{1,n})_G=\sum_{r\in \mathcal{R}} (\mu_{1,\nu^n},{}^r\mu_{1,\nu^n })_{K\cap {{}^rK}},$$
where $\mathcal{R}$ is a complete set of representatives for the double cosets of $K$ in $G$. As $\mathcal{R}$ we can choose the set $\{s(\alpha), \overline{s}(\beta)\mid \alpha,\beta\in F^\times\}$, where 
$$s=s(\alpha)=\left(
\begin{array}{cc}
\alpha & 0 \\
0 & 1/\alpha
\end{array}
\right), \quad \overline{s}=\overline{s}(\beta)=\left(
\begin{array}{cc}
0 & -1/\beta \\
\beta & 0
\end{array}\right) \in S.$$
Note that $|KsK|=q^4$ and $|K\overline{s}K|=q^5$. Since the $\mu_{1,\nu^n}$'s are linear characters, it suffices to show that for $r\neq s(1)$, the restrictions of $\mu_{1,\nu^n}$ and ${}^r\mu_{1,\nu^n}$ to $K\cap {}^rK$ are distinct. 

First, we look at the double cosets $K\overline{s}(\beta)K$. For all ${}^{\overline{s}} k \in K \cap {}^{\overline{s}}K$, we have $\mu_{1,\nu^n}({}^{\overline{s}}k) =\lambda_{\nu^n}(\beta x)$ and ${}^{\overline{s}}\mu_{1,\nu^n}({}^{\overline{s}}k) =\mu_{1,\nu^n}(k)= \lambda_{\nu^n}(y)$. It follows that, if $\mu_{1,\nu^n}={}^{\overline{s}}\mu_{1,\nu^n}$, then $\lambda_{\nu^n}(\beta x)=\lambda_{\nu^n}(y)$, for all $x,y\in F$. In particular, for $x=0$, we have $\lambda_{\nu^n}(y)=1$ for all $y\in F$, i.e. $Ker(\lambda_{\nu^n})= \mathbf{Z}(H_1(q))$, forcing $\nu^n=0$, a contradiction.

Next, we look at the double cosets $Ks(\alpha)K$. For all ${}^s k \in K \cap {}^sK$, we have $\mu_{1,\nu^n}({}^sk)=\lambda_1(a\alpha^2)\lambda_{\nu^n}(\frac{y}{\alpha})$ and ${}^s\mu_{1,\nu^n}({}^s k)=\mu_{1,\nu^n}(k)= \lambda_1(a)\lambda_{\nu^n}(y)$. 
It follows that, if $\mu_{1,\nu^n}={}^s\mu_{1,\nu^n}$, then $\lambda_1(a\alpha^2)\lambda_{\nu^n}(\frac{y}{\alpha})=\lambda_1(a)\lambda_{\nu^n}(y)$, for all $a,y\in F$. In particular, for $y=0$, we get $\lambda_{\alpha^2}=\lambda_1$, and so $\alpha=1$. Clearly, for $\alpha=1$, the two restrictions are the same character. This proves that the characters $\kappa_{1,n}$ are irreducible. 

In the same way, we can prove that the characters $\kappa_{\nu,n}$ are also irreducible.  To conclude, we are left to show that the characters $\kappa_{1,n}$ and $\kappa_{\nu,n}$ are pairwise distinct. This can be obtained proving that $(\kappa_{d,n},\kappa_{d_1,n_1})_G=0$, for $d,d_1\in \{1,\nu\}$, $1\leq n\leq \frac{q-1}{2}$ and $(d,n)\neq (d_1,n_1)$.
As above, we exploit Mackey's formula. The double cosets $K\overline{s}(\beta)K$ are dealt with in the same way as before.
In the case of the double cosets $Ks(\alpha)K$, for $d=d_1$ we can argue as before. In the case $(d,d_1)=(1,\nu)$, if the restrictions of $\mu_{1,\nu^n}$ and $\mu_{\nu,\nu^{n_1}}$ are the same, then
$$\lambda_1(a\alpha^2)\lambda_{\nu^n}(\frac{y}{\alpha})=\lambda_\nu(a)\lambda_{\nu^{n_1}}(y)$$
for all $a,y\in F$. In particular, for $y=0$, we get $\lambda_{\alpha^2}=\lambda_\nu$, a contradiction, since $\nu$ is not a square in $F$.\\

In conclusion, the desired character table of $G$ can be described as follows:
\newpage

\begin{small}
\begin{center}
\begin{tabular}{|c||c|c|c|c|c|} \hline
{} & $1$ & $\mathscr{A}(z)$ & $\mathscr{B}$ & $\mathscr{C}(0)$ & $\mathscr{D}_k(0)$  \\ \hline \hline  
$1_G$ & 1 & 1 & 1 & 1  & 1   \\
$\eta_1$ & $\frac{q-1}{2}$ & $\frac{q-1}{2}$  &  $\frac{q-1}{2}$ & $\frac{-\delta(q-1)}{2}$ & 0  \\
$\eta_2$ & $\frac{q-1}{2}$ & $\frac{q-1}{2}$  &  $\frac{q-1}{2}$ & $\frac{-\delta(q-1)}{2}$ & 0  \\
$\xi_1$ & $\frac{q+1}{2}$ & $\frac{q+1}{2}$  &  $\frac{q+1}{2}$ & $\frac{\delta(q+1)}{2}$ & $(-1)^k$  \\
$\xi_2$ & $\frac{q+1}{2}$ & $\frac{q+1}{2}$  &  $\frac{q+1}{2}$ & $\frac{\delta(q+1)}{2}$ & $(-1)^k$ \\
$\theta_j$ & $q-1$ & $q-1$ & $q-1$ & $(-1)^j(q-1)$ & $0$  \\ 
$\psi$ & $q$ & $q$ & $q$ & $q$ & $1$ \\ 
$\chi_i$ & $q+1$ & $q+1$ & $q+1$ & $(-1)^i(q+1)$ & $\rho^{ik}+\rho^{-ik}$ \\ \hline
$\kappa_0$ & $q^2-1$ & $q^2-1$ & $-1$ & $0$ & $0$ \\
$\kappa_{1,n}$ & $q^2-1$ & $q^2-1$ & $-1$ & $0$ & $0$   \\ 
$\kappa_{\nu,n}$ & $q^2-1$ & $q^2-1$ & $-1$ & $0$ & $0$   \\ \hline
$\omega_u$ & $q$ & $q\lambda_u(z)$ & $0$ & $\delta$ & $(-1)^k$\\
$\omega_u\eta_1$ & $\frac{q(q-1)}{2}$ & $\frac{q(q-1)\lambda_u(z)}{2}$ & $0$ & $-\frac{(q-1)}{2}$ & $0$  \\
$\omega_u\eta_2$ & $\frac{q(q-1)}{2}$ & $\frac{q(q-1)\lambda_u(z)}{2}$ & $0$ & $-\frac{(q-1)}{2}$ & $0$  \\
$\omega_u\xi_1$ & $\frac{q(q+1)}{2}$ & $\frac{q(q+1)\lambda_u(z)}{2}$ & $0$ & $\frac{(q+1)}{2}$ & $1$ \\
$\omega_u\xi_2$ & $\frac{q(q+1)}{2}$ & $\frac{q(q+1)\lambda_u(z)}{2}$ & $0$ & $\frac{(q+1)}{2}$ & $1$ \\
$\omega_u\theta_j$ & $q(q-1)$ & $q(q-1)\lambda_u(z)$ & $0$ & $(-1)^j\delta (q-1)$ & $0$ \\
$\omega_u\psi$ & $q^2$ & $q^2 \lambda_u(z)$ & $0$ & $\delta q$ & $(-1)^k$  \\
$\omega_u\chi_i$ & $q(q+1)$ & $q(q+1)\lambda_u(z)$ & $0$ & $(-1)^i\delta (q+1)$ & $(-1)^k(\rho^{ik}+\rho^{-ik})$ \\ \hline
\end{tabular}

\begin{tabular}{|c||c|c|c|c|} \hline
{} &  $\mathscr{E}(0)$ & $\mathscr{F}(0)$ & $\mathscr{G}_m(0)$ & $\mathscr{H}(0)$\\ \hline\hline 
$1_G$ & 1 & 1 & 1 & 1  \\
$\eta_1$ & $\frac{-\delta(-1+\sqrt{\delta q})}{2}$ & $\frac{-\delta(-1-\sqrt{\delta q})}{2}$ & $(-1)^{m+1}$ & $\frac{(-1+\sqrt{\delta q})}{2}$\\
$\eta_2$  & $\frac{-\delta(-1-\sqrt{\delta q})}{2}$ & $\frac{-\delta(-1+\sqrt{\delta q})}{2}$ & $(-1)^{m+1}$ & $\frac{(-1-\sqrt{\delta q})}{2}$  \\
$\xi_1$  & $\frac{\delta(1+\sqrt{\delta q})}{2}$ & $\frac{\delta(1-\sqrt{\delta q})}{2}$ & $0$ & $\frac{(1+\sqrt{\delta q})}{2}$  \\
$\xi_2$ & $\frac{\delta(1-\sqrt{\delta q})}{2}$  & $\frac{\delta(1+\sqrt{\delta q})}{2}$ & $0$ & $\frac{(1-\sqrt{\delta q})}{2}$ \\
$\theta_j$ & $(-1)^{j+1}$ & $(-1)^{j+1}$ & $-(\sigma^{jm}+\sigma^{-jm})$ & $-1$  \\
$\psi$ & $0$  & $0$ & $-1$ & $0$  \\
$\chi_i$ & $(-1)^i$ &  $(-1)^i$ & $0$ & $1$ \\  \hline
$\kappa_0$  & $0$ & $0$ & $0$ & $q-1$ \\ 
$\kappa_{1,n}$  & $0$ & $0$ & $0$ & $-1+(\frac{2}{F})Q(\lambda)$ \\ 
$\kappa_{\nu,n}$  & $0$ & $0$ & $0$ & $-1-(\frac{2}{F})Q(\lambda)$ \\ \hline
$\omega_u$  & $\delta$ &  $\delta$ & $(-1)^{m+1}$ & $Q(\lambda_u)$  \\
$\omega_u\eta_1$  & $\frac{1-\sqrt{\delta q}}{2}$  & $\frac{1+\sqrt{\delta q}}{2}$  & $1$  & $\frac{(-1+\sqrt{\delta q})Q(\lambda_u)}{2}$  \\
$\omega_u\eta_2$ & $\frac{1+\sqrt{\delta q}}{2}$ & $\frac{1-\sqrt{\delta q}}{2}$  & $1$  & $\frac{(-1-\sqrt{\delta q})Q(\lambda_u)}{2}$  \\
$\omega_u\xi_1$ & $\frac{1+\sqrt{\delta q}}{2}$ & $\frac{1-\sqrt{\delta q}}{2}$ & $0$  & $\frac{(1+\sqrt{\delta q})Q(\lambda_u)}{2}$\\
$\omega_u\xi_2$ & $\frac{1-\sqrt{\delta q}}{2}$ & $\frac{1+\sqrt{\delta q}}{2}$ & $0$  & $\frac{(1-\sqrt{\delta q})Q(\lambda_u)}{2}$\\ 
$\omega_u\theta_j$  & $(-1)^{j+1}\delta$ & $(-1)^{j+1}\delta$ & $(-1)^m(\sigma^{jm}+\sigma^{-jm})$ & $-Q(\lambda_u)$ \\
$\omega_u\psi$  & $0$ & $0$ & $(-1)^m$ & $0$ \\
$\omega_u\chi_i$ & $(-1)^{i}\delta$ & $(-1)^{i}\delta$ & $0$ & $Q(\lambda_u)$ \\ \hline
\end{tabular}

\begin{tabular}{|c||c|c|c|} \hline
{} & $\mathscr{I}(0)$  & $\mathscr{L}_m$ & $\mathscr{M}_m$ \\\hline\hline 
$1_G$ & 1 & 1 & 1\\
$\eta_1$ & $\frac{(-1-\sqrt{\delta q})}{2}$ & $\frac{(-1+\sqrt{\delta q})}{2}$ & $\frac{(-1-\sqrt{\delta q})}{2}$\\
$\eta_2$ & $\frac{(-1+\sqrt{\delta q})}{2}$ & $\frac{(-1-\sqrt{\delta q})}{2}$ & $\frac{(-1+\sqrt{\delta q})}{2}$\\
$\xi_1$ & $\frac{(1-\sqrt{\delta q})}{2}$ & $\frac{(1+\sqrt{\delta q})}{2}$ & $\frac{(1-\sqrt{\delta q})}{2}$\\
$\xi_2$ & $\frac{(1+\sqrt{\delta q})}{2}$ & $\frac{(1-\sqrt{\delta q})}{2}$ & $\frac{(1+\sqrt{\delta q})}{2}$\\
$\theta_j$ & $-1$ & $-1$ & $-1$\\
$\psi$ & $0$ & $0$ & $0$ \\
$\chi_i$ & $1$  & $1$ & $1$ \\ \hline
$\kappa_0$ & $q-1$  & $-1$ & $-1$\\ 
$\kappa_{1,n}$ & $-1-(\frac{2}{F})Q(\lambda)$ & $\sum_{t \in F^\times}\lambda(-\frac{t^3+\nu^{n+m}}{t})$ &   $\sum_{t\in F^\times}\lambda(-\frac{\nu t^3+\nu^{n+m}}{t})$ \\
$\kappa_{\nu,n}$ & $-1+(\frac{2}{F})Q(\lambda)$ & $\sum_{t \in F^\times}\lambda(-\frac{\nu t^3+\nu^{n+m}}{t})$ &   $\sum_{t\in F^\times}\lambda(-\frac{\nu^2 t^3+\nu^{n+m}}{t})$ \\\hline
$\omega_u$ & $-Q(\lambda_u)$ & $0$ & $0$ \\
$\omega_u\eta_1$ & $\frac{(1+\sqrt{\delta q})Q(\lambda_u)}{2}$ & $0$ & $0$ \\
$\omega_u\eta_2$ & $\frac{(1-\sqrt{\delta q})Q(\lambda_u)}{2}$ & $0$ & $0$ \\
$\omega_u\xi_1$ & $\frac{(-1+\sqrt{\delta q})Q(\lambda_u)}{2}$ & $0$ & $0$ \\
$\omega_u\xi_2$ & $\frac{(-1-\sqrt{\delta q})Q(\lambda_u)}{2}$ & $0$ & $0$ \\ 
$\omega_u\theta_j$ & $Q(\lambda_u)$ & $0$ & $0$ \\
$\omega_u\psi$  & $0$ & $0$ & $0$ \\
$\omega_u\chi_i$ & $-Q(\lambda_u)$  & $0$ & $0$ \\ \hline
\end{tabular}
\end{center}
\end{small}

\noindent \textbf{Notations.}  \quad $1\leq i, k \leq \frac{q-3}{2}$, $1\leq j, m,n \leq \frac{q-1}{2}$.
$\delta=(-1)^{\frac{q-1}{2}}$, $\rho=e^{\frac{2\pi \imath}{q-1}}$, $\sigma=e^{\frac{2\pi \imath}{q+1}}$. $F=GF(q)$, 
$F^\times=\langle\nu\rangle$, $u \in F^\times$. $\lambda$ is a (fixed) non-trivial character of $\mathbf{Z}(G)$.
$\lambda_u$ is the linear character of $\mathbf{Z}(G)$ defined by $\lambda_u(z)=\lambda(uz)$ for all $z \in \mathbf{Z}(G)$. 
$$Q(\lambda)=\sum_{t\in F}\lambda(-t^2/2),\quad \quad Q(\lambda_u)=\Big(\dfrac{u}{F}\Big)Q(\lambda) .$$

For all $\chi \in Irr(G)$, we have (but this is omitted from the Table)
$$\chi(\mathscr{C}(z))=\frac{\chi(\mathscr{A}(z))}{\chi(\mathscr{A}(0))}\chi(\mathscr{C}(0)),$$
and likewise for the other conjugacy classes.

\end{document}